\documentclass[11pt]{amsart}

\usepackage[T1]{fontenc}
\usepackage[utf8]{inputenc}
\usepackage{lmodern}
\usepackage{amsmath,amssymb,amsthm,mathtools}
\usepackage{enumitem}
\usepackage{microtype}
\usepackage[colorlinks=true,linkcolor=blue,citecolor=blue,urlcolor=blue]{hyperref}
\usepackage{geometry}
\numberwithin{equation}{section}

\newtheorem{theorem}{Theorem}[section]
\newtheorem{proposition}[theorem]{Proposition}
\newtheorem{lemma}[theorem]{Lemma}
\newtheorem{corollary}[theorem]{Corollary}

\theoremstyle{definition}
\newtheorem{definition}[theorem]{Definition}
\theoremstyle{remark}
\newtheorem{remark}[theorem]{Remark}

\newcommand{\fPD}{\operatorname{fPD}}
\newcommand{\pd}{\operatorname{pd}}
\newcommand{\Ext}{\operatorname{Ext}}
\newcommand{\Hom}{\operatorname{Hom}}
\newcommand{\Tor}{\operatorname{Tor}}
\newcommand{\Spec}{\operatorname{Spec}}
\newcommand{\Max}{\operatorname{Max}}
\newcommand{\GSpec}{\operatorname{GSpec}}
\newcommand{\Ann}{\operatorname{Ann}}
\newcommand{\depth}{\operatorname{depth}}
\newcommand{\htp}{\operatorname{ht}}
\newcommand{\Kgr}{\operatorname{K.grade}}
\newcommand{\Egr}{\operatorname{E.grade}}
\newcommand{\FPR}{\mathcal{FPR}}
\newcommand{\SpecOne}{\operatorname{X}^{1}}

\title[Goldman primes and cyclic presentations]
{Goldman primes and higher cyclic presentations for the small finitistic dimension}

\author{Xiaolei Zhang}
\address{School of Mathematics and Statistics, Tianshui Normal University, Tianshui 741001, China}
\email{zxlrghj@163.com}

\author{Hwankoo Kim}
\address{Division of Computer Engineering, Hoseo University, Asan 31499, Republic of Korea}
\email{hkkim@hoseo.edu}

\hypersetup{%
  pdftitle={Goldman primes and higher cyclic presentations for the small finitistic dimension},
  pdfauthor={Xiaolei Zhang and Hwankoo Kim},
  pdfsubject={Small finitistic dimension, polynomial extensions, and cyclic weak (n,d)-rings},
  pdfkeywords={small finitistic dimension, polynomial ring, Koszul grade, Goldman prime, Jacobson ring, cyclic weak (n,d)-ring, n-presented module, idealization}
}

\begin{document}

\begin{abstract}
Let $R$ be a commutative ring and let $\fPD(R)$ be its small finitistic dimension in the sense of Glaz.  We study two ways in which Koszul grade can escape a restricted homological test.  For a polynomial extension, we prove that if $R$ is Noetherian, the maximal ideals of $R[X]$ detect exactly the Goldman primes of $R$, and this gives the formula
\[
 \fPD(R[X])
 =1+\sup\{\depth R_{\mathfrak p}\mid \mathfrak p\in\GSpec(R)\}.
\]
We construct Noetherian local rings $R_m$ with $\fPD(R_m)=0$ and $\fPD(R_m[X])=m+1$, and we show that Koszul grade can increase by two along an adjacent pair of primes.  We then consider cyclic weak $(n,d)$-conditions.  For every $n\geq2$ and $h\geq1$, we construct a local idealization $T$ such that every $n$-presented cyclic $T$-module is projective, while $\fPD(T)=h$.  The ring $T$ has no nonzero proper finitely presented ideals, but it has a finitely generated ideal $I$ with $\Ext_T^i(T/I,T)=0$ for $i<h$ and $\Ext_T^h(T/I,T)\neq0$.  Thus the cyclic condition controls $\fPD$ at presentation level one, but at no higher level.
\end{abstract}

\maketitle

\section{Introduction}\label{sec:introduction}

The small finitistic dimension records the finite projective dimensions that can be realized by modules having finite resolutions by finitely generated projective modules.  This finiteness condition is essential outside the Noetherian setting: a finitely generated module of finite projective dimension need not have finitely generated syzygies.  Following Glaz \cite{Glaz}, let $\FPR(R)$ denote the class of $R$-modules admitting a finite projective resolution by finitely generated projective modules, and set
\[
 \fPD(R)=\sup\{\pd_R M\mid M\in\FPR(R)\}.
\]
This invariant belongs to the family of finitistic dimensions initiated by Bass \cite{Bass}, but its finite-resolution definition makes it particularly well adapted to commutative non-Noetherian homological algebra.

A decisive feature of $\fPD$ is that it can be read from Koszul grade.  More precisely, $\fPD(R)$ is the supremum of the Koszul grades of the maximal ideals of $R$ \cite{ZhangWang,ZhangII}.  The present paper asks how much of that grade remains visible after one changes the ring or restricts the class of cyclic modules used as tests.  Two natural problems lead to sharply different kinds of hidden grade.

The first problem is the behavior of $\fPD$ under adjoining an indeterminate.  The regular element $X$ forces $\fPD(R[X])\geq\fPD(R)+1$, but an upper bound cannot be obtained merely by looking at maximal ideals of $R$: a maximal ideal of $R[X]$ may contract to a nonmaximal prime.  Kaplansky's characterization of G-domains shows that the possible contractions are exactly the Goldman primes.  Consequently, the appropriate geometric detector is not $\Max(R)$ but the Goldman locus $\GSpec(R)$.

The second problem comes from the hierarchy of $n$-presented modules introduced in the study of $(n,d)$-rings.  Costa's $(n,d)$-rings test all $n$-presented modules, Zhou's weak version uses flat dimension, and Mahdou's cyclic version tests only $n$-presented cyclic modules \cite{Costa, Mahdou, Zhou}.  At level $n=1$, every quotient $R/I$ with $I$ finitely generated is visible.  At level $n=2$, the defining ideal must itself be finitely presented.  This additional requirement permits finitely generated ideals of large Koszul grade to disappear completely from the cyclic test class.

The first main theorem determines the polynomial behavior and exhibits the size of the possible jump.

\begin{theorem}\label{thm:intro-polynomial}
Let $R$ be a commutative ring and let $X$ be an indeterminate.
\begin{enumerate}[label=\textup{(\roman*)}]
\item \cite[Proposition 4.1]{ZhangII} One has $\fPD(R[X])\geq\fPD(R)+1$.
\item \cite[Theorem 4.2]{ZhangII} If $R$ is Jacobson, then $\fPD(R[X])=\fPD(R)+1$.
\item If $R$ is Noetherian, then
\[
 \fPD(R[X])
 =1+\sup\{\depth R_{\mathfrak p}\mid \mathfrak p\in\GSpec(R)\}.
\]
\item For every integer $N\geq1$, there is a Noetherian local ring $R$ such that $\fPD(R)=0$ and $\fPD(R[X])=N$.
\end{enumerate}
\end{theorem}

The local examples in part (iv) show that Noetherianity and finiteness of $\fPD(R)$ do not by themselves force a one-step increase.  They also produce an adjacent pair of primes $\mathfrak P\subsetneq\mathfrak Q$ for which $\Kgr(\mathfrak P)=0$ and $\Kgr(\mathfrak Q)=2$.  Thus the height-one relation between two primes does not impose a corresponding one-step bound on Koszul grade.

The second main theorem gives the exact contrast between the first and higher presentation levels.

\begin{theorem}\label{thm:intro-cyclic}
Let $n\geq2$ and $h\geq1$.  There exists a commutative local ring $T$ such that every $n$-presented cyclic $T$-module is projective and $\fPD(T)=h$.  Moreover, $T$ has a finitely generated proper ideal $I$ satisfying
\[
 \Ext_T^i(T/I,T)=0\quad(0\leq i<h),
 \qquad
 \Ext_T^h(T/I,T)\neq0.
\]
By contrast, every cyclic weak $(1,d)$-ring has small finitistic dimension at most $d$.
\end{theorem}

The construction behind Theorem~\ref{thm:intro-cyclic} is a large square-zero idealization.  Countably many copies of every height-one quotient provide fresh coordinates that prevent nonzero proper ideals from being finitely presented.  At the same time, those height-one quotients retain a prescribed Koszul grade.  This separates presentation-theoretic visibility from homological depth and answers the question in \cite[Open Question~3.8]{ZhangIII} negatively for every $n\geq2$.

The two parts of the paper are unified by a visibility principle.  In the polynomial problem, maximal ideals upstairs reveal nonmaximal Goldman primes downstairs.  In the cyclic problem, the $n$-presentation requirement hides finitely generated ideals whose Ext grade controls $\fPD$.  In both cases, an apparently natural restricted test misses precisely the objects carrying the larger grade.

Section~\ref{sec:preliminaries} develops the required grade formalism and recalls $n$-presented modules, idealizations, and Goldman primes.  Section~\ref{sec:polynomial} proves the polynomial formulas and constructs the Noetherian counterexamples.  Section~\ref{sec:cyclic} develops the large idealizations and establishes the sharp threshold between $n=1$ and $n\geq2$.  All rings are commutative with identity, all modules are unital, and $\infty+1=\infty$.

\section{Grade detection and auxiliary constructions}\label{sec:preliminaries}

This section places the two problems in a common framework.  We first recall the grade description of the small finitistic dimension, then isolate the finite-presentation fact needed for cyclic quotients, and finally identify the primes that arise as contractions of maximal ideals in a polynomial ring.

\subsection{Koszul grade and the small finitistic dimension}

Let $A$ be a ring, let $N$ be an $A$-module, and let $\mathbf a=a_1,\ldots,a_s$ be a finite sequence.  Write $K_\bullet(\mathbf a;A)$ for the Koszul complex and put $K^\bullet(\mathbf a;N)=\Hom_A(K_\bullet(\mathbf a;A),N)$.  If $I=(\mathbf a)$, the Koszul grade of $I$ on $N$ is
\[
 \Kgr_A(I,N)
 =\inf\{i\geq0\mid H^i(K^\bullet(\mathbf a;N))\neq0\}.
\]
The value is independent of the chosen finite generating sequence.  For an arbitrary ideal $J$, one defines
\[
 \Kgr_A(J,N)
 =\sup\{\Kgr_A(I,N)\mid I\subseteq J\text{ is finitely generated}\}.
\]
We use $\inf\varnothing=\infty$.  If $I$ is finitely generated, its Ext grade on $N$ is $\Egr_A(I,N)=\inf\{i\geq0\mid\Ext_A^i(A/I,N)\neq0\}$.  These notions and their non-Noetherian properties are developed in \cite{AsgharzadehTousi}; see also \cite[Sections~1.6 and~9.1]{BrunsHerzog}.

The following standard properties will be used throughout.

\begin{proposition}\label{prop:grade-properties}
Let $A$ be a ring, let $N$ be an $A$-module, and let $I\subseteq J$ be ideals.
\begin{enumerate}[label=\textup{(\roman*)}]
\item If $x\in J$ is $N$-regular, then
\[
 \Kgr_A(J,N)
 =1+\Kgr_{A/xA}\bigl(J(A/xA),N/xN\bigr).
\]
\item If $A\to B$ is a ring homomorphism and $N$ is a $B$-module, then $\Kgr_A(J,N)=\Kgr_B(JB,N)$.
\item If $A\to B$ is faithfully flat, then $\Kgr_A(J,N)=\Kgr_B(JB,N\otimes_A B)$.
\item One has $\Kgr_A(I,N)\leq\Kgr_A(J,N)$.
\item If $I$ is finitely generated, then $\Kgr_A(I,N)=\Egr_A(I,N)$.
\item If $J$ is proper, then $\Kgr_A(J,N)=\Kgr_A(\mathfrak p,N)$ for some prime ideal $\mathfrak p\supseteq J$.
\end{enumerate}
\end{proposition}

\begin{proof}
Parts (i)--(iv) are the regular-sequence, change-of-rings, faithfully flat base-change, and monotonicity properties of Koszul grade.  Prime attainment in (vi) follows from the corresponding result for polynomial grade, which agrees with Koszul grade.  Part (v) is the equality of Ext and Koszul grade for finitely generated ideals.  See \cite[Propositions~2.2 and~2.3]{AsgharzadehTousi} for all these statements.
\end{proof}

The connection with the small finitistic dimension is the principal detection theorem used below.

\begin{theorem}\label{thm:fPD-grade}
For every commutative ring $A$,
\begin{equation}\label{eq:fPD-grade}
 \fPD(A)=\sup\{\Kgr_A(\mathfrak m,A)\mid \mathfrak m\in\Max(A)\}.
\end{equation}
If $A$ is Noetherian, then
\begin{equation}\label{eq:fPD-depth}
 \fPD(A)=\sup\{\depth A_{\mathfrak m}\mid \mathfrak m\in\Max(A)\}.
\end{equation}
\end{theorem}

\begin{proof}
Formula \eqref{eq:fPD-grade} is the Koszul-grade form of the characterization in \cite[Theorem~3.1]{ZhangWang}; see also \cite[Theorem~3.4]{ZhangII}.  In the Noetherian case, the Koszul grade of a maximal ideal equals the depth of the corresponding local ring \cite[Section~1.2]{BrunsHerzog}.
\end{proof}

A complementary Ext criterion will distinguish the cases $n=1$ and $n\geq2$.

\begin{theorem}\label{thm:Ext-detection}
Let $A$ be a commutative ring and let $d\geq0$.  Then $\fPD(A)\leq d$ if and only if, for every finitely generated ideal $I$ of $A$, the vanishing $\Ext_A^i(A/I,A)=0$ for $0\leq i\leq d$ implies $\Ext_A^i(A/I,A)=0$ for all $i\geq0$.
\end{theorem}

\begin{proof}
This is \cite[Theorem~2.5]{ZhangIII}.
\end{proof}

\subsection{Higher presentations and idealizations}

The cyclic test considered later depends on how many finite syzygies a quotient possesses.  We therefore fix the presentation terminology explicitly.

\begin{definition}\label{def:n-presented}
Let $n\geq0$.  An $A$-module $E$ is \emph{$n$-presented} if there is an exact sequence
\[
 F_n\longrightarrow F_{n-1}\longrightarrow\cdots\longrightarrow F_0
 \longrightarrow E\longrightarrow0
\]
in which every $F_i$ is finitely generated free.  Thus $0$-presented means finitely generated and $1$-presented means finitely presented.
\end{definition}

Costa calls $A$ an $(n,d)$-ring when every $n$-presented module has projective dimension at most $d$ \cite{Costa}.  Zhou's weak version replaces projective dimension by flat dimension \cite{Zhou}.  To avoid ambiguity between these notions, we use the following explicit name for the cyclic condition introduced by Mahdou \cite{Mahdou}.

\begin{definition}\label{def:cyclic-weak}
A ring $A$ is a \emph{cyclic weak $(n,d)$-ring} if $\pd_A E\leq d$ for every $n$-presented cyclic $A$-module $E$.
\end{definition}

The defining ideal of a cyclic module acquires one fewer finite syzygy.

\begin{lemma}\label{lem:cyclic-ideal}
Let $A$ be a ring and let $I$ be an ideal of $A$.  If $A/I$ is $n$-presented for some $n\geq2$, then $I$ is finitely presented.
\end{lemma}

\begin{proof}
It is enough to consider $n=2$.  Choose an exact sequence $F_2\to F_1\to F_0\to A/I\to0$ with the $F_i$ finitely generated free, and put $K=\ker(F_0\to A/I)$.  The induced sequence $F_2\to F_1\to K\to0$ is a finite presentation of $K$.  Comparing $0\to I\to A\to A/I\to0$ with $0\to K\to F_0\to A/I\to0$, Schanuel's lemma gives $I\oplus F_0\cong K\oplus A$.  Hence $I$ is a direct summand of a finitely presented module and is finitely presented by \cite[Theorem 2.6.5]{WK24}.
\end{proof}

The examples in both main sections use the idealization construction.  If $D$ is a ring and $E$ is a $D$-module, the \emph{idealization} $D\ltimes E$ is the abelian group $D\oplus E$ with multiplication
\[
 (a,e)(b,f)=(ab,af+be).
\]
The subset $0\ltimes E$ is a square-zero ideal.  An element $(a,e)$ is a unit precisely when $a$ is a unit of $D$, every prime ideal is $\mathfrak p\ltimes E$ for a unique $\mathfrak p\in\Spec(D)$, and $\Spec(D\ltimes E)\to\Spec(D)$ is an order-preserving homeomorphism.  In particular, $D\ltimes E$ is local when $D$ is local.  These and further structural properties are recorded in \cite{AndersonWinders, ElKKM}.

\subsection{Goldman primes and maximal polynomial ideals}

The polynomial problem is governed by a locus that is usually larger than the maximal spectrum.  We recall its definition before using it.

\begin{definition}\label{def:Goldman}
Let $D$ be an integral domain with quotient field $K$.  The domain $D$ is a \emph{G-domain}, or \emph{Goldman domain}, if $K$ is a finitely generated $D$-algebra.  Equivalently, $K=D[1/t]$ for some nonzero $t\in D$.  A prime ideal $\mathfrak p$ of a ring $R$ is a \emph{Goldman prime} if $R/\mathfrak p$ is a G-domain.  We write
\[
 \GSpec(R)=\{\mathfrak p\in\Spec(R)\mid R/\mathfrak p\text{ is a G-domain}\}.
\]
\end{definition}

The equivalence in the domain definition follows by multiplying the denominators of finitely many algebra generators.  The connection with polynomial maximal ideals is due to Goldman and Kaplansky \cite{Goldman,Kaplansky}.

\begin{proposition}\label{prop:Goldman-contraction}
For a prime ideal $\mathfrak p$ of a ring $R$, the following conditions are equivalent.
\begin{enumerate}[label=\textup{(\roman*)}]
\item $\mathfrak p\in\GSpec(R)$.
\item There is a maximal ideal $\mathfrak M$ of $R[X]$ such that $\mathfrak M\cap R=\mathfrak p$.
\end{enumerate}
\end{proposition}

\begin{proof}
Put $D=R/\mathfrak p$, and let $\pi:R[X]\twoheadrightarrow D[X]$ be the natural map.  Kaplansky's characterization says that a domain $D$ is a G-domain if and only if $D[X]$ has a maximal ideal $\mathfrak N$ with $\mathfrak N\cap D=0$; see \cite[Theorem~24]{Kaplansky}.

If $D$ is a G-domain, choose such an ideal $\mathfrak N$ and put $\mathfrak M=\pi^{-1}(\mathfrak N)$.  Then $\mathfrak M$ is maximal and $\mathfrak M\cap R=\mathfrak p$.  Conversely, if $\mathfrak M\in\Max(R[X])$ contracts to $\mathfrak p$, then $\mathfrak pR[X]\subseteq\mathfrak M$, and $\mathfrak N=\mathfrak M/\mathfrak pR[X]$ is maximal in $D[X]$ with
\[
 \mathfrak N\cap D=(\mathfrak M\cap R)/\mathfrak p=0.
\]
Kaplansky's criterion now shows that $D$ is a G-domain.
\end{proof}

A ring is \emph{Jacobson} if every prime ideal is the intersection of the maximal ideals containing it.  Equivalently, every maximal ideal of $R[X]$ contracts to a maximal ideal of $R$; see \cite{Goldman,Kaplansky}.

\begin{corollary}\label{cor:GSpec-Jacobson}
Every maximal ideal of $R$ belongs to $\GSpec(R)$, and $R$ is Jacobson if and only if $\GSpec(R)=\Max(R)$.
\end{corollary}

\begin{proof}
The quotient by a maximal ideal is a field and hence a G-domain.  By Proposition~\ref{prop:Goldman-contraction}, the equality $\GSpec(R)=\Max(R)$ is equivalent to the assertion that every maximal ideal of $R[X]$ contracts to a maximal ideal of $R$, which is the polynomial characterization of Jacobson rings.
\end{proof}

\section{Polynomial extensions and the Goldman locus}\label{sec:polynomial}

This section determines what a maximal ideal of $R[X]$ contributes to the small finitistic dimension.  The regular variable gives a universal lower bound.  When contractions are maximal the bound is exact, while in the Noetherian case the full answer is obtained by allowing all Goldman contractions.

\subsection{The universal lower bound and the Jacobson equality}

The variable itself produces one additional unit of Koszul grade.

\begin{theorem}\label{thm:polynomial-lower}  \cite[Proposition 4.1]{ZhangII}
For every commutative ring $R$, one has $\fPD(R[X])\geq\fPD(R)+1$.
\end{theorem}

\begin{proof}
Set $S=R[X]$.  For $\mathfrak m\in\Max(R)$, the ideal $\mathfrak M=\mathfrak mS+XS$ is maximal, and $X$ is $S$-regular.  Proposition~\ref{prop:grade-properties} gives
\[
 \Kgr_S(\mathfrak M,S)
 =1+\Kgr_R(\mathfrak m,R).
\]
Taking the supremum over $\mathfrak m$ and applying Theorem~\ref{thm:fPD-grade} proves the inequality.
\end{proof}

When every maximal ideal upstairs lies over a maximal ideal downstairs, the same calculation becomes an equality at each point.

\begin{theorem}\label{thm:polynomial-Jacobson}  \cite[Theorem 4.2]{ZhangII}
Let $R$ be Jacobson.  If $\mathfrak M\in\Max(R[X])$ and $\mathfrak m=\mathfrak M\cap R$, then
\begin{equation}\label{eq:pointwise-Jacobson}
 \Kgr_{R[X]}(\mathfrak M,R[X])
 =1+\Kgr_R(\mathfrak m,R).
\end{equation}
Consequently, $\fPD(R[X])=\fPD(R)+1$.
\end{theorem}

\begin{proof}
Put $S=R[X]$.  Since $R$ is Jacobson, $\mathfrak m$ is maximal.  The maximal ideal $\mathfrak M/\mathfrak mS$ of $(R/\mathfrak m)[X]$ is generated by a monic irreducible polynomial.  Choose a monic lift $f\in S$.  Then $\mathfrak M=\mathfrak mS+fS$, the element $f$ is $S$-regular, and $B=S/fS$ is a nonzero finite free, hence faithfully flat, $R$-algebra.  Moreover, $\mathfrak M/fS=\mathfrak mB$.  The regular-element shift and faithfully flat base change give
\[
 \Kgr_S(\mathfrak M,S)
 =1+\Kgr_B(\mathfrak mB,B)
 =1+\Kgr_R(\mathfrak m,R).
\]
This proves \eqref{eq:pointwise-Jacobson}.  Formula \eqref{eq:fPD-grade} now yields the global equality.
\end{proof}

\subsection{The exact Noetherian formula}

For a Noetherian base, the local polynomial depth is always one more
than the depth at the contracted prime.  We give first a proof using the
Koszul-grade machinery developed above and then a direct proof from the
flat local depth formula.

\begin{lemma}\label{lem:depth-poly-local}
Let $R$ be a Noetherian ring, set $S=R[X]$, and let
$\mathfrak M\in\Max(S)$.  If $\mathfrak p=\mathfrak M\cap R$, then
\begin{equation}\label{eq:depth-poly-local}
 \depth S_{\mathfrak M}=\depth R_{\mathfrak p}+1.
\end{equation}
\end{lemma}

\begin{proof}
Set $A=R_{\mathfrak p}$ and let
$\mathfrak N=\mathfrak M A[X]$.  Since
$R\setminus\mathfrak p$ is disjoint from $\mathfrak M$, the ideal
$\mathfrak N$ is maximal and
\[
 S_{\mathfrak M}\cong A[X]_{\mathfrak N}.
\]
Moreover, $\mathfrak N\cap A=\mathfrak pA$, and
$\mathfrak N/\mathfrak pA[X]$ is a maximal ideal of
$\kappa(\mathfrak p)[X]$.  Choose a monic polynomial
$f\in A[X]$ whose image generates
$\mathfrak N/\mathfrak pA[X]$.  Then
\[
 \mathfrak N=\mathfrak pA[X]+fA[X].
\]

The monic polynomial $f$ is regular on $A[X]$.  Put
$C=A[X]/fA[X]$.  Then $C$ is a nonzero finite free, and hence
faithfully flat, $A$-algebra, and
$\mathfrak N/fA[X]=\mathfrak pC$.  Since $\mathfrak N$ and
$\mathfrak pA$ are maximal ideals and all the rings involved are
Noetherian, the regular-element shift and faithfully flat invariance
of Koszul grade give
\[
 \begin{aligned}
 \depth A[X]_{\mathfrak N}
 &=\Kgr_{A[X]}(\mathfrak N,A[X])\\
 &=1+\Kgr_C(\mathfrak pC,C)\\
 &=1+\Kgr_A(\mathfrak pA,A)\\
 &=1+\depth A.
 \end{aligned}
\]
Since $A=R_{\mathfrak p}$ and
$A[X]_{\mathfrak N}\cong S_{\mathfrak M}$, the result follows.
\end{proof}

\begin{proof}[Alternative proof]
Set
\[
 A=R_{\mathfrak p}
 \qquad\text{and}\qquad
 B=S_{\mathfrak M}.
\]
The induced homomorphism $A\to B$ is a flat local homomorphism of
Noetherian local rings.  Indeed, it is obtained from the flat map
$R\to R[X]$ by localization, and the inverse image of the maximal
ideal $\mathfrak M B$ is $\mathfrak pA$.

Let $\mathfrak N=\mathfrak M A[X]$ and put
\[
 \mathfrak n=\mathfrak N/\mathfrak pA[X].
\]
Recall that
$\kappa(\mathfrak p)=R_{\mathfrak p}/\mathfrak pR_{\mathfrak p}
\cong\operatorname{Frac}(R/\mathfrak p)$
is the residue field of $R$ at $\mathfrak p$.
Then $\mathfrak n$ is a maximal ideal of
$\kappa(\mathfrak p)[X]$, and the closed fiber of $A\to B$ is
\[
 B/\mathfrak pAB
 \cong
 \kappa(\mathfrak p)[X]_{\mathfrak n}.
\]
Since $\mathfrak n$ is generated by a nonconstant irreducible
polynomial, this closed fiber is a one-dimensional regular local
ring.  Consequently,
\[
 \depth(B/\mathfrak pAB)=1.
\]
The flat local depth formula
\cite[Lemma~10.163.2, Tag~0337]{Stacks}
therefore yields
\[
 \depth B
 =\depth A+\depth(B/\mathfrak pAB)
 =\depth A+1.
\]
Thus
\[
 \depth S_{\mathfrak M}
 =\depth R_{\mathfrak p}+1.  \qedhere
\]
\end{proof}

The Goldman contraction criterion now turns the pointwise calculation into a global formula.

\begin{theorem}\label{thm:Goldman-depth}
Let $R$ be Noetherian.  Then
\begin{equation}\label{eq:Goldman-depth}
 \fPD(R[X])
 =1+\sup\{\depth R_{\mathfrak p}\mid \mathfrak p\in\GSpec(R)\}.
\end{equation}
\end{theorem}

\begin{proof}
By Theorem~\ref{thm:fPD-grade} and Lemma~\ref{lem:depth-poly-local},
\[
 \fPD(R[X])
 =1+\sup\{\depth R_{\mathfrak M\cap R}\mid \mathfrak M\in\Max(R[X])\}.
\]
Proposition~\ref{prop:Goldman-contraction} identifies the set of contractions with $\GSpec(R)$.
\end{proof}

The usual one-step formula is therefore equivalent to a precise depth bound on the Goldman locus.

\begin{corollary}\label{cor:one-step-criterion}
Let $R$ be Noetherian.  Then $\fPD(R[X])=\fPD(R)+1$ if and only if $\depth R_{\mathfrak p}\leq\fPD(R)$ for every $\mathfrak p\in\GSpec(R)$.  In particular, the equality holds for every Noetherian Jacobson ring.
\end{corollary}

\begin{proof}
Combine \eqref{eq:Goldman-depth} with \eqref{eq:fPD-depth}.  The final assertion follows from Corollary~\ref{cor:GSpec-Jacobson}.
\end{proof}

\subsection{Arbitrarily large local jumps}

A finite idealization can force depth zero at the unique maximal ideal while preserving large depth at a nonmaximal Goldman prime.  This produces every possible positive value of $\fPD(R[X])$ over a Noetherian local base of small finitistic dimension zero.

\begin{theorem}\label{thm:Noetherian-jumps}
Let $k$ be a field and let $m\geq1$.  Put
\[
 D_m=k[u_1,\ldots,u_m,v]_{(u_1,\ldots,u_m,v)},
 \qquad
 E_m=D_m/(u_1,\ldots,u_m,v),
 \qquad
 R_m=D_m\ltimes E_m.
\]
Then $R_m$ is Noetherian local and
\[
 \fPD(R_m)=0,
 \qquad
 \fPD(R_m[X])=m+1.
\]
\end{theorem}

\begin{proof}
Let $\mathfrak m=(u_1,\ldots,u_m,v)D_m$.  Since $D_m$ is Noetherian and $E_m$ is finitely generated, $R_m$ is Noetherian.  It is local with maximal ideal $\mathfrak n=\mathfrak m\ltimes E_m$.  The nonzero element $e=(0,1)$ is annihilated by $\mathfrak n$, so $\depth R_m=0$.  Formula \eqref{eq:fPD-depth} gives $\fPD(R_m)=0$.

Let $\mathfrak q=(u_1,\ldots,u_m)D_m$ and $\mathfrak p=\mathfrak q\ltimes E_m$.  Then $R_m/\mathfrak p\cong k[v]_{(v)}$, and its quotient field is obtained by inverting $v$.  Hence $\mathfrak p\in\GSpec(R_m)$.  Since $(v,0)\notin\mathfrak p$ and $vE_m=0$, localization kills the square-zero summand and gives $(R_m)_{\mathfrak p}\cong(D_m)_{\mathfrak q}$.  The latter is regular local of depth $m$.

It remains to bound the depth at every other prime.  Each prime of $R_m$ has the form $\mathfrak s\ltimes E_m$ with $\mathfrak s\in\Spec(D_m)$.  The depth at $\mathfrak n$ is zero.  If $\mathfrak s\neq\mathfrak m$, choose $a\in\mathfrak m\setminus\mathfrak s$.  The element $a$ annihilates $E_m$ and becomes a unit after localization, so
\[
 (R_m)_{\mathfrak s\ltimes E_m}\cong(D_m)_{\mathfrak s}.
\]
Because $D_m$ is regular local of dimension $m+1$ and $\mathfrak s$ is not maximal, $\depth(D_m)_{\mathfrak s}=\htp(\mathfrak s)\leq m$.  Thus the supremum in \eqref{eq:Goldman-depth} is exactly $m$, and $\fPD(R_m[X])=m+1$.
\end{proof}

\begin{corollary}\label{cor:all-polynomial-values}
For every integer $N\geq1$, there is a Noetherian local ring $R$ such that $\fPD(R)=0$ and $\fPD(R[X])=N$.
\end{corollary}

\begin{proof}
Take a field when $N=1$, and take $R_{N-1}$ from Theorem~\ref{thm:Noetherian-jumps} when $N\geq2$.
\end{proof}

\begin{remark}\label{rem:scope-of-plus-one}
Theorem~\ref{thm:Noetherian-jumps} shows that neither Noetherianity nor the finiteness of $\fPD(R)$ is sufficient for the equality $\fPD(R[X])=\fPD(R)+1$.  In particular, the Noetherian assertion in \cite[Theorem~4.4]{ZhangII} and the finite-$\fPD$ assertion in \cite[Theorem~3.4]{XiongEtAl} require additional hypotheses.  The Goldman-depth formula identifies the missing geometric datum.  For the latter argument, restriction of scalars also presents a basic obstruction: an $R[X]$-module in $\FPR(R[X])$ need not belong to $\FPR(R)$ as an $R$-module.
\end{remark}

\subsection{Adjacent primes do not impose an adjacent grade bound}

The preceding examples also isolate a local obstruction to a height-theoretic proof of the one-step formula.  Koszul grade can rise by two even when prime height rises by one.

\begin{proposition}\label{prop:adjacent-grade}
There is a Noetherian ring $S$ with prime ideals $\mathfrak P\subsetneq\mathfrak Q$ such that $\htp(\mathfrak Q)=\htp(\mathfrak P)+1$, but
\[
 \Kgr_S(\mathfrak Q,S)>\Kgr_S(\mathfrak P,S)+1.
\]
More precisely, one may arrange $\Kgr_S(\mathfrak P,S)=0$ and $\Kgr_S(\mathfrak Q,S)=2$.
\end{proposition}

\begin{proof}
Let $D=k[u,v]_{(u,v)}$, let $R=D\ltimes k$, and put $S=R[X]$.  Set
\[
 \mathfrak p=(u)D\ltimes k,
 \qquad
 \mathfrak P=\mathfrak pS,
 \qquad
 \mathfrak Q=\mathfrak pS+(vX-1)S.
\]
The quotient $S/\mathfrak Q$ is $k[v]_{(v)}[X]/(vX-1)\cong k(v)$, so $\mathfrak Q$ is maximal.

Let $e=(0,1)\in R$.  Every element of $\mathfrak p$ annihilates $e$, and hence $0\neq e\in\Ann_S(\mathfrak P)$.  Thus $\Kgr_S(\mathfrak P,S)=0$.  Since $ve=0$, one has $e=-e(vX-1)$, and therefore $\mathfrak Q=(u,vX-1)S$.

Multiplication by $vX-1$ on $S$ is injective.  Indeed, coefficient comparison in $(vX-1)\sum_{i=0}^t a_iX^i=0$ first gives $a_0=0$ and then successively $a_1=\cdots=a_t=0$.  Moreover, $S/(vX-1)S\cong R_v\cong D_v$, because localization at $v$ kills the square-zero summand.  The image of $u$ is regular in the domain $D_v$.  Hence $vX-1,u$ is an $S$-regular sequence, so $\Kgr_S(\mathfrak Q,S)\geq2$.

The nilpotent ideal $(0\ltimes k)S$ is contained in every prime, and quotienting by it identifies $\Spec(S)$ with $\Spec(D[X])$.  Under this identification, $\mathfrak P$ corresponds to $(u)D[X]$, so $\htp(\mathfrak P)=1$.  Since $\mathfrak Q$ is generated by two elements, Krull's height theorem gives $\htp(\mathfrak Q)\leq2$, while the strict chain from the unique minimal prime through $\mathfrak P$ to $\mathfrak Q$ gives the reverse inequality.  Thus $\htp(\mathfrak Q)=2$.  Grade is at most height in a Noetherian ring, and hence $\Kgr_S(\mathfrak Q,S)=2$; see \cite[Chapter~1]{BrunsHerzog} or \cite[Chapter~13]{Matsumura}.
\end{proof}

\begin{remark}\label{rem:adjacent-warning}
Proposition~\ref{prop:adjacent-grade} disproves the inequality $\Kgr_A(\mathfrak q,A)\leq\Kgr_A(\mathfrak p,A)+1$ for adjacent primes $\mathfrak p\subsetneq\mathfrak q$.  A regular sequence chosen inside $\mathfrak p$ controls the quotient by that sequence, not the quotient by the whole prime $\mathfrak p$; the two operations cannot be interchanged.
\end{remark}

\section{Higher cyclic presentations and large idealizations}\label{sec:cyclic}

This section shows that the higher cyclic presentation tests can be homologically empty even when the ring has large small finitistic dimension.  The construction first eliminates all nonzero proper finitely presented ideals and then uses the same square-zero module to retain a prescribed Koszul grade.

\subsection{Eliminating nontrivial finitely presented ideals}

Let $(A,\mathfrak m)$ be a local Noetherian UFD that is not a field, and write
\[
 \SpecOne(A)=\{\mathfrak p\in\Spec(A)\mid\htp(\mathfrak p)=1\}.
\]
Define
\begin{equation}\label{eq:large-idealization}
 M_A=\bigoplus_{\mathfrak p\in\SpecOne(A)}(A/\mathfrak p)^{(\mathbb N)},
 \qquad
 T_A=A\ltimes M_A.
\end{equation}
Every element of $M_A$ has finite support.  The countably many copies of each height-one quotient will allow a new coordinate to be chosen outside any finite family of proposed relations.

The first lemma selects a height-one quotient on which a nonzero proper ideal acquires a nontrivial first Tor class.

\begin{lemma}\label{lem:height-one-Tor}
Let $A$ be a Noetherian UFD and let $0\neq J\subsetneq A$ be an ideal.  Then $\Tor_1^A(A/J,A/\mathfrak p)\neq0$ for some $\mathfrak p\in\SpecOne(A)$.
\end{lemma}

\begin{proof}
Choose $0\neq a\in J$ for which the number of irreducible factors, counted with multiplicity, is minimal.  Since $J$ is proper, $a$ is not a unit.  Write $a=\pi b$ with $\pi$ irreducible.  Minimality gives $b\notin J$.  The ideal $\mathfrak p=(\pi)$ is prime because $A$ is a UFD, and it has height one by Krull's principal ideal theorem \cite[Chapter~11]{AtiyahMacdonald}.

Tensoring $0\to J\to A\to A/J\to0$ with $A/\mathfrak p$ gives the standard identification
\[
 \Tor_1^A(A/J,A/\mathfrak p)
 \cong (J\cap\mathfrak p)/(J\mathfrak p);
\]
see, for example, \cite[Chapter~7]{Rotman}.  The class of $a$ is nonzero: if $a\in J\mathfrak p=\pi J$, cancellation would imply $b\in J$.
\end{proof}

The support-separation argument now prevents finite presentation of every nonzero proper ideal.

\begin{theorem}\label{thm:no-finitely-presented-ideals}
Let $(A,\mathfrak m)$ be a local Noetherian UFD that is not a field, and let $T_A$ be defined by \eqref{eq:large-idealization}.  Then $T_A$ has no nonzero proper finitely presented ideals.
\end{theorem}

\begin{proof}
Let $0\neq L\subsetneq T_A$ be finitely generated.  Choose generators $z_i=(a_i,u_i)$ for $1\leq i\leq s$, and put $J=(a_1,\ldots,a_s)\subseteq A$.  The ideal $J$ is proper: if $J=A$, an $A$-linear combination of the $z_i$ has first coordinate $1$ and is a unit of $T_A$.

Let $\Phi:T_A^s\to L$ be the canonical surjection with $\Phi(e_i)=z_i$, and write $K=\ker\Phi$.  We prove that $K$ is not finitely generated.

\smallskip
\noindent\emph{Case 1: $J\neq0$.}
Let $\psi:A^s\to J$ be defined by $\psi(e_i)=a_i$, and put $S=\ker\psi$.  By Lemma~\ref{lem:height-one-Tor}, choose $\mathfrak p\in\SpecOne(A)$ such that, for $C=A/\mathfrak p$, one has $\Tor_1^A(A/J,C)\neq0$.  Tensoring the exact sequences $0\to J\to A\to A/J\to0$ and $0\to S\to A^s\to J\to0$ with $C$ gives
\begin{equation}\label{eq:Tor-obstruction}
 \frac{\ker(C^s\xrightarrow{(a_1,\ldots,a_s)}C)}
 {\operatorname{Im}(S\otimes_A C\to C^s)}
 \cong\Tor_1^A(A/J,C)\neq0.
\end{equation}
Choose $g=(g_1,\ldots,g_s)$ in the numerator of \eqref{eq:Tor-obstruction} but outside the denominator.

Suppose that $K$ is generated by finitely many vectors $w_j=(b^{(j)},v^{(j)})\in A^s\oplus M_A^s$.  Comparing first coordinates in $\Phi(w_j)=0$ gives $b^{(j)}\in S$.  Each $v^{(j)}$ has finite support, so one may choose a copy $C_0\cong C$ among the countably many $C$-summands of $M_A$ that occurs in none of these supports.  Regard $g$ as a vector in $M_A^s$ supported on $C_0^s$.  Since $\sum_i a_i g_i=0$ in $C_0$, the vector $(0,g)$ belongs to $K$.

For coefficients $(\alpha_j,h_j)\in T_A$, the $C_0$-component of the square-zero part of $\sum_j(\alpha_j,h_j)w_j$ is $\sum_j b^{(j)}(h_j)_{C_0}$.  It belongs to $\operatorname{Im}(S\otimes_A C_0\to C_0^s)$, whereas the $C_0$-component of $(0,g)$ is the chosen vector $g$ outside that image.  This contradiction proves that $K$ is not finitely generated.

\smallskip
\noindent\emph{Case 2: $J=0$.}
Then $L=0\ltimes N$ for a nonzero finitely generated $A$-submodule $N$ of $M_A$.  Choose a minimal generating set $u_1,\ldots,u_s$ of $N$ and put $S=\ker(A^s\to N)$.  Since $A$ is local, Nakayama's lemma gives $S\subseteq\mathfrak mA^s$; see \cite[Chapter~2]{AtiyahMacdonald}.  The map $\Phi$ satisfies $\Phi(b,v)=(0,\sum_i b_i u_i)$, so $(b,v)\in K$ implies $b\in S$, while every $(0,v)$ belongs to $K$.

A height-one prime of $A$ exists: choose a nonzero nonunit and an irreducible factor.  If $K$ were generated by finitely many $w_j=(b^{(j)},v^{(j)})$, choose $\mathfrak p\in\SpecOne(A)$ and a copy $C_0\cong A/\mathfrak p$ outside all supports of the $v^{(j)}$.  Every $C_0$-component generated by the $w_j$ lies in $(\mathfrak mC_0)^s$, because each $b^{(j)}$ lies in $S\subseteq\mathfrak mA^s$.  On the other hand, $K$ contains a square-zero vector supported on $C_0^s$ whose first coordinate is $1$ and whose other coordinates are zero.  This is impossible because $1\notin\mathfrak mC_0$.

Thus $K$ is not finitely generated in either case.  If $L$ were finitely presented, the kernel of every surjection from a finite free module onto $L$ would be finitely generated, by Schanuel's lemma.  The chosen surjection $\Phi$ contradicts this, so $L$ is not finitely presented.
\end{proof}

The absence of finitely presented proper ideals makes every higher-presented cyclic module trivial.

\begin{corollary}\label{cor:higher-cyclic-projective}
For every $n\geq2$, each $n$-presented cyclic $T_A$-module is projective.  Hence $T_A$ is a cyclic weak $(n,0)$-ring for every $n\geq2$.
\end{corollary}

\begin{proof}
Write an $n$-presented cyclic module as $T_A/L$.  Lemma~\ref{lem:cyclic-ideal} makes $L$ finitely presented, and Theorem~\ref{thm:no-finitely-presented-ideals} forces $L=0$ or $L=T_A$.  Thus the quotient is $T_A$ or $0$.
\end{proof}

\begin{remark}\label{rem:fresh-coordinate}
The countable multiplicity in \eqref{eq:large-idealization} is essential to this proof.  After finitely many proposed relation generators have been fixed, one must choose a copy of the relevant quotient outside all of their supports.  A single copy of each $A/\mathfrak p$ would not provide that fresh coordinate.
\end{remark}

\subsection{Retaining a prescribed Koszul grade}

We now specialize to a regular local base.  Every height-one quotient then has depth one less than the base, and the Koszul complex splits across the idealization.

Fix $r\geq2$ and a field $k$.  Put
\begin{equation}\label{eq:regular-large-idealization}
 A_r=k[x_1,\ldots,x_r]_{\mathfrak m},
 \quad \mathfrak m=(x_1,\ldots,x_r),
 \quad
 M_r=\bigoplus_{\mathfrak p\in\SpecOne(A_r)}(A_r/\mathfrak p)^{(\mathbb N)},
 \quad
 T_r=A_r\ltimes M_r.
\end{equation}
The ring $A_r$ is a regular local UFD of dimension $r$, while $T_r$ is local with maximal ideal $\mathfrak n=\mathfrak m\ltimes M_r$.  Moreover, $\dim T_r=r$ because $\Spec(T_r)\cong\Spec(A_r)$.

For $1\leq i\leq r$, set $X_i=(x_i,0)$ and let $I_r=(X_1,\ldots,X_r)T_r$.

The height-one summands determine the first nonzero Koszul cohomology of $I_r$.

\begin{proposition}\label{prop:grade-Ir}
One has $\Kgr_{T_r}(I_r,T_r)=r-1$.  Equivalently,
\[
 \Ext_{T_r}^i(T_r/I_r,T_r)=0\quad(0\leq i<r-1),
 \qquad
 \Ext_{T_r}^{r-1}(T_r/I_r,T_r)\neq0.
\]
\end{proposition}

\begin{proof}
As an $A_r$-module, $T_r=A_r\oplus M_r$, and multiplication by $X_i$ acts as multiplication by $x_i$ on both summands.  Hence the Koszul cochain complex splits as
\begin{equation}\label{eq:Koszul-splitting}
 K^\bullet_{T_r}(\mathbf X;T_r)
 \cong K^\bullet_{A_r}(\mathbf x;A_r)
 \oplus K^\bullet_{A_r}(\mathbf x;M_r),
\end{equation}
where $\mathbf x=x_1,\ldots,x_r$.

The sequence $\mathbf x$ is a regular system of parameters, so the first complex in \eqref{eq:Koszul-splitting} has no nonzero cohomology below degree $r$.  If $\mathfrak p\in\SpecOne(A_r)$, the UFD property gives $\mathfrak p=(\pi)$ for a nonzerodivisor $\pi$.  Thus $A_r/\mathfrak p$ is a Cohen--Macaulay hypersurface of dimension $r-1$, and
\[
 \Kgr_{A_r}(\mathfrak m,A_r/\mathfrak p)=r-1;
\]
see \cite{BrunsHerzog,Eisenbud}.  Since the Koszul complex is finite free, $\Hom$ from it commutes with direct sums, and cohomology commutes with these direct sums.  Therefore the first nonzero cohomology of $K^\bullet_{A_r}(\mathbf x;M_r)$ occurs in degree $r-1$.

It follows from \eqref{eq:Koszul-splitting} that $\Kgr_{T_r}(I_r,T_r)=r-1$.  Since $I_r$ is finitely generated, Proposition~\ref{prop:grade-properties}(v) gives the asserted Ext vanishing and nonvanishing.
\end{proof}

The same ideal determines the small finitistic dimension exactly, not merely from below.

\begin{theorem}\label{thm:exact-cyclic-example}
For every $r\geq2$, the local ring $T_r$ in \eqref{eq:regular-large-idealization} has the following properties.
\begin{enumerate}[label=\textup{(\roman*)}]
\item Every $n$-presented cyclic $T_r$-module is projective for every $n\geq2$.
\item The ideal $I_r$ satisfies
\[
 \Ext_{T_r}^i(T_r/I_r,T_r)=0\quad(0\leq i<r-1),
 \qquad
 \Ext_{T_r}^{r-1}(T_r/I_r,T_r)\neq0.
\]
\item One has $\fPD(T_r)=r-1$.
\end{enumerate}
\end{theorem}

\begin{proof}
Part (i) follows from Corollary~\ref{cor:higher-cyclic-projective}, and part (ii) is Proposition~\ref{prop:grade-Ir}.

For part (iii), every prime of $T_r$ has the form $\mathfrak q\ltimes M_r$.  A prime containing $I_r$ must contain each $(x_i,0)$, so $\mathfrak q\supseteq\mathfrak m$ and hence $\mathfrak q=\mathfrak m$.  Thus the unique prime containing $I_r$ is $\mathfrak n$.  By the prime-attainment property in Proposition~\ref{prop:grade-properties}(vi),
\[
 \Kgr_{T_r}(\mathfrak n,T_r)
 =\Kgr_{T_r}(I_r,T_r)=r-1.
\]
Since $T_r$ is local, formula \eqref{eq:fPD-grade} yields $\fPD(T_r)=r-1$.
\end{proof}

Every positive value of the small finitistic dimension therefore occurs among rings whose higher cyclic tests are all projective.

\begin{corollary}\label{cor:prescribed-fPD}
For every $n\geq2$ and every $h\geq1$, there is a commutative local cyclic weak $(n,0)$-ring $T$ such that $\fPD(T)=h$.
\end{corollary}

\begin{proof}
Take $T=T_{h+1}$ in Theorem~\ref{thm:exact-cyclic-example}.
\end{proof}

\begin{corollary}\label{cor:negative-Mahdou-question}
For every $n\geq2$ and $d\geq0$, there is a cyclic weak $(n,d)$-ring $R$ with $\fPD(R)>d$.  Hence \cite[Open Question~3.8]{ZhangIII} has a negative answer.
\end{corollary}

\begin{proof}
Take $R=T_{d+2}$.  It is cyclic weak $(n,0)$ and hence cyclic weak $(n,d)$, while $\fPD(R)=d+1$.
\end{proof}

\subsection{The threshold at the first presentation level}

At presentation level one, every finitely generated ideal remains visible through its quotient.  The Ext-detection theorem then restores control of the small finitistic dimension.

\begin{proposition}\label{prop:level-one}\cite[Proposition 3.7]{ZhangIII}
If $R$ is a cyclic weak $(1,d)$-ring, then $\fPD(R)\leq d$.
\end{proposition}

\begin{proof}
Let $I$ be a finitely generated ideal.  Then $R/I$ is finitely presented and cyclic, so $\pd_R(R/I)\leq d$.  Consequently, $\Ext_R^i(R/I,R)=0$ for all $i>d$.  If these groups also vanish for $0\leq i\leq d$, they vanish in every degree.  Theorem~\ref{thm:Ext-detection} gives $\fPD(R)\leq d$.
\end{proof}

The preceding results give a sharp dichotomy.

\begin{corollary}\label{cor:sharp-threshold}
The cyclic weak condition controls the small finitistic dimension at level $n=1$, but at no level $n\geq2$.  More precisely, a cyclic weak $(1,d)$-ring satisfies $\fPD(R)\leq d$, whereas for each fixed $n\geq2$ the values of $\fPD$ on cyclic weak $(n,0)$-rings include every positive integer.
\end{corollary}

\begin{remark}\label{rem:visibility-summary}
For $n\geq2$, a finitely generated ideal $I$ makes $R/I$ only $1$-presented; it need not make the quotient $n$-presented.  In the rings $T_r$, this gap is maximal: no nonzero proper finitely generated ideal is finitely presented, yet the ideals $I_r$ have arbitrarily large Ext grade.  This is the presentation-theoretic counterpart of the polynomial phenomenon, where a nonmaximal Goldman prime carries depth invisible on the maximal spectrum of the base ring.
\end{remark}

\end{document}